\documentclass[11pt]{amsart}
\usepackage{amsmath, amssymb, amscd, amsthm, mathrsfs, url ,pinlabel,hyperref, verbatim, lipsum, wrapfig}
\usepackage[text={6.5in,9.5in}, centering, letterpaper, dvips]{geometry}
\usepackage{color,multirow,dcpic,latexsym,pictexwd,graphicx,pdfpages}
\usepackage{fancyhdr, tikz, setspace, tensor, enumitem, dirtytalk}
\usetikzlibrary{positioning, arrows.meta}
\usepackage{scrextend}

\definecolor{maroon}{rgb}{0.5, 0.0, 0.0}

\hypersetup{
  colorlinks   = true, 
  urlcolor     = maroon, 
  linkcolor    = maroon, 
  citecolor   = maroon 
}

\usepackage[T1]{fontenc}
\usepackage{lmodern}
\DeclareUrlCommand{\bfurl}{}

\newtheorem {theorem}{Theorem}
\newtheorem {lemma}[theorem]{Lemma}
\newtheorem {definition}[theorem]{Definition}

\newtheorem {conjecture}[theorem]{Conjecture}

\newtheorem {problem}[theorem]{Problem}
\theoremstyle{remark}

\newtheorem {remark}[theorem]{Remark}

\numberwithin{equation}{section}
\numberwithin{theorem}{section}

\title[Plumbed Homology Spheres Bounding Contractible Manifolds and Homology Balls]{Classical and New Plumbed Homology Spheres Bounding Contractible Manifolds and Homology Balls}
\author{O{\u{g}}uz \c{S}avk}
\address{CNRS and Laboratorie de Math\'ematiques Jean Leray, Nantes Universit\'e, 44322 Nantes, France}
\email{\url{oguz.savk@cnrs.fr}, \url{oguz.savk@univ-nantes.fr}}
\urladdr{\url{https://sites.google.com/view/oguzsavk/}}
\date{}

\begin{document}

\begin{abstract}
A central problem in low-dimensional topology asks which homology $3$-spheres bound contractible $4$-manifolds or homology $4$-balls. In this paper, we address this question for plumbed $3$-manifolds and we present two new infinite families. We consider most of the classical examples from around the nineteen eighties by reproving that they all bound Mazur manifolds. We also show that several well-known families bound possibly different types of $4$-manifolds, called Po\'enaru homology $4$-balls. To unify classical and new results in a fairly simple way, we modify Mazur's argument and work with Po\'enaru manifolds.
\end{abstract}
\maketitle

\section{Introduction}
\label{intro}
Freedman's breakthrough work \cite{Fre82} expresses that every integral homology $3$-sphere bounds a topological contractible $4$-manifold. However, the smooth analogue of this implication produces an unresolved problem in low-dimensional topology.

\begin{problem}[Problem 4.2, \cite{K78b}] 
\label{prob}
Which integral homology $3$-spheres bound smooth contractible $4$-manifolds or smooth integral homology $4$-balls?\footnote{In this paper, we always work with integral homology $3$-spheres, smooth contractible $4$-manifolds, and smooth integral homology $4$-balls, hence we drop the \emph{integral}, \emph{smooth}, $3$- and $4$- prefixes.}
\end{problem}

Currently, there is no complete program for the resolution of this crucial problem. However, in the past four decades, a considerable amount of progress has been achieved by restricting to plumbed $3$-manifolds, and especially Seifert fibered homology spheres with three singular fibers. In general, Seifert fibered homology spheres can be uniquely realized as the boundaries of negative-definite, minimal (i.e. no more blow down on the branches), and unimodular plumbing graphs of $4$-manifolds. These graphs have one node (i.e., a vertex connecting more than two branches) with adjacent branches and the number of branches corresponds to the number of singular fibers. In this paper, we always work with minimal, negative-definite and unimodular plumbing graphs.

Several infinite families of Seifert fibered homology spheres with three singular fibers are known to bound contractible manifolds. After Kirby's celebrated work \cite{K78}, the classical articles appeared subsequently: Akbulut and Kirby \cite{AK79}, Casson and Harer \cite{CH81}, Stern \cite{S78}, Fintushel and Stern \cite{FS81}, Maruyama \cite{M81}, \cite{M82}, and Fickle \cite{F84}. In addition, some of these results were found independently of Kirby calculus, see Fukuhara \cite{F78} and Martin \cite{M79}. One can weaken Problem~\ref{prob} to ask which homology spheres bound rational homology balls. Surprisingly, the latter problem still remains difficult. For affirmative answers and constructions, see the papers \cite{FS84, AL18, S20, Sim20}. 

When the number of singular fibers increases, there is a bold conjecture for Seifert fibered homology spheres bounding homology balls first indicated by Fintushel and Stern \cite{FS87}, and explicitly stated by Koll\'ar \cite{K08}. Thus, the above plenty number of examples may seem sporadic.

\begin{conjecture}[Three Fibers Conjecture\footnote{The proposed name for the conjecture comes from an e-mail correspondence with Ronald Fintushel.}]
\label{conj}
A Seifert fibered homology sphere with more than three singular fibers cannot bound a homology ball.
\end{conjecture}

The main obstruction comes from the Fintushel-Stern $R$-invariant \cite{FS85}. There are other important obstructions; for instance, Rokhlin $\mu$-invariant \cite{R52} and Ozsv\'ath-Szab\'o $d$-invariant \cite{OS03}. Together with the short-cut of Neumann and Zagier \cite{NZ85}, we know that Seifert fibered homology spheres with the node different than minus one weight do not bound homology balls. This condition is even enough for the obstruction from bounding rational homology balls, see the article of Issa and McCoy \cite{IM18}.

From the perspective of $4$-dimensional handlebodies, the simplest contractible manifolds after the $4$-ball $B^4$ are \emph{Mazur manifolds} \cite{M61}. They are contractible manifolds obtained by attaching a single $1$- and $2$-handle to $B^4$. In this paper, we provide a modification of Mazur's argument by increasing the number of $1$- and $2$-handles through the eyes of knot theory. In the original construction, we change the role of the unknot with any ribbon knot; therefore, we call the resulting spaces \emph{Po\'enaru manifolds} \cite{P60}, see Section~\ref{ihb} for details.

We first notice that Conjecture~\ref{conj} cannot be generalized for plumbed homology spheres that are not Seifert fibered. The first examples were provided by Maruyama \cite{M82}, independently refound by Akbulut and Karakurt \cite{AK14}. These manifolds are not Seifert fibered\footnote{\label{xxx} Note that $\partial X(1) = \Sigma(2,5,7)$ and $\partial X'(1) = \Sigma(3,4,5)$, compare with \cite{AK79}, \cite{CH81}, and \cite{S20}.} unless $n=1$ and they can be realized as the boundary of a plumbing graph with two nodes and five branches, and are reproven via our approach.

\begin{theorem}[Maruyama, Akbulut-Karakurt]
\label{maruyama}
Let $X(n)$ be the Maruyama plumbed $4$-manifold in the left-hand side of Figure~\ref{fig:plumb}. Then for each $n\geq 1$, its boundary $\partial X(n)$ is a homology sphere which bounds a Mazur type contractible manifold with one $0$-handle, one $1$-handle and one $2$-handle.
\end{theorem}

\begin{figure}[ht]
\begin{center}
\includegraphics[width=0.7\columnwidth]{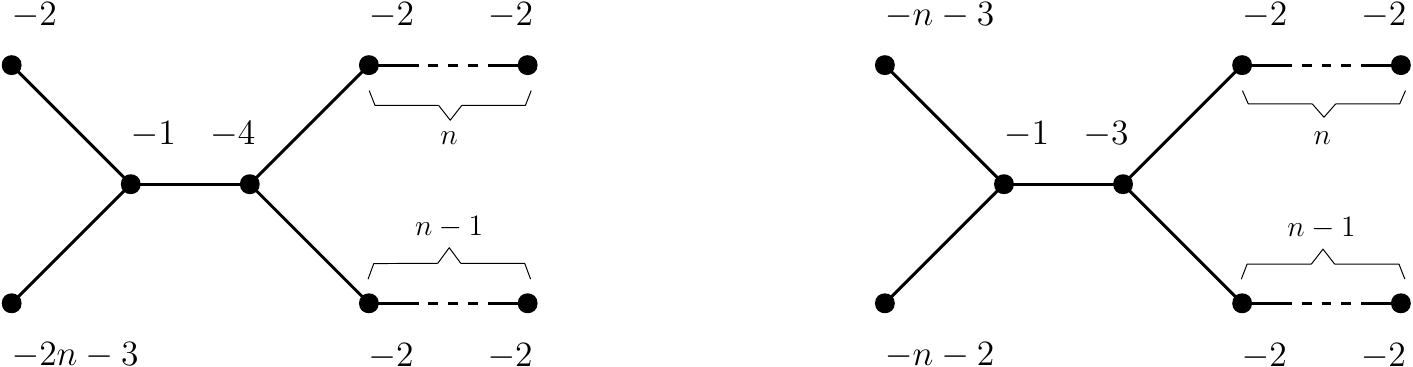}
\end{center}
\caption{Maruyama plumbed $4$-manifold $X(n)$ and its companion $X'(n)$.}
\label{fig:plumb}
\end{figure}

Next, we present a new infinite family with the same number of nodes and branches, which are not Seifert fibered\footref{xxx} again unless $n=1$.

\begin{theorem}
\label{maruyama2}
Let $X'(n)$ be the companion of Maruyama plumbed $4$-manifold in the right-hand side of Figure~\ref{fig:plumb}. Then for each $n\geq 1$, its boundary $\partial X'(n)$ is a homology sphere which bounds a Mazur type contractible manifold with one $0$-handle, one $1$-handle and one $2$-handle.
\end{theorem}

We exhibit one more new infinite family by increasing the complexity of the graph whose plumbing has three nodes and seven branches.

\begin{theorem}
\label{main}
Let $W(n)$ be the Ramanujam plumbed $4$-manifold in Figure~\ref{fig:plumb2}. Then for each $n\geq 1$, its boundary $\partial W(n)$ is a homology sphere which bounds a Po\'enaru type homology ball with one $0$-handle, two $1$-handles and two $2$-handles.
\end{theorem}

\begin{figure}[ht]
\begin{center}
\includegraphics[width=0.45\columnwidth]{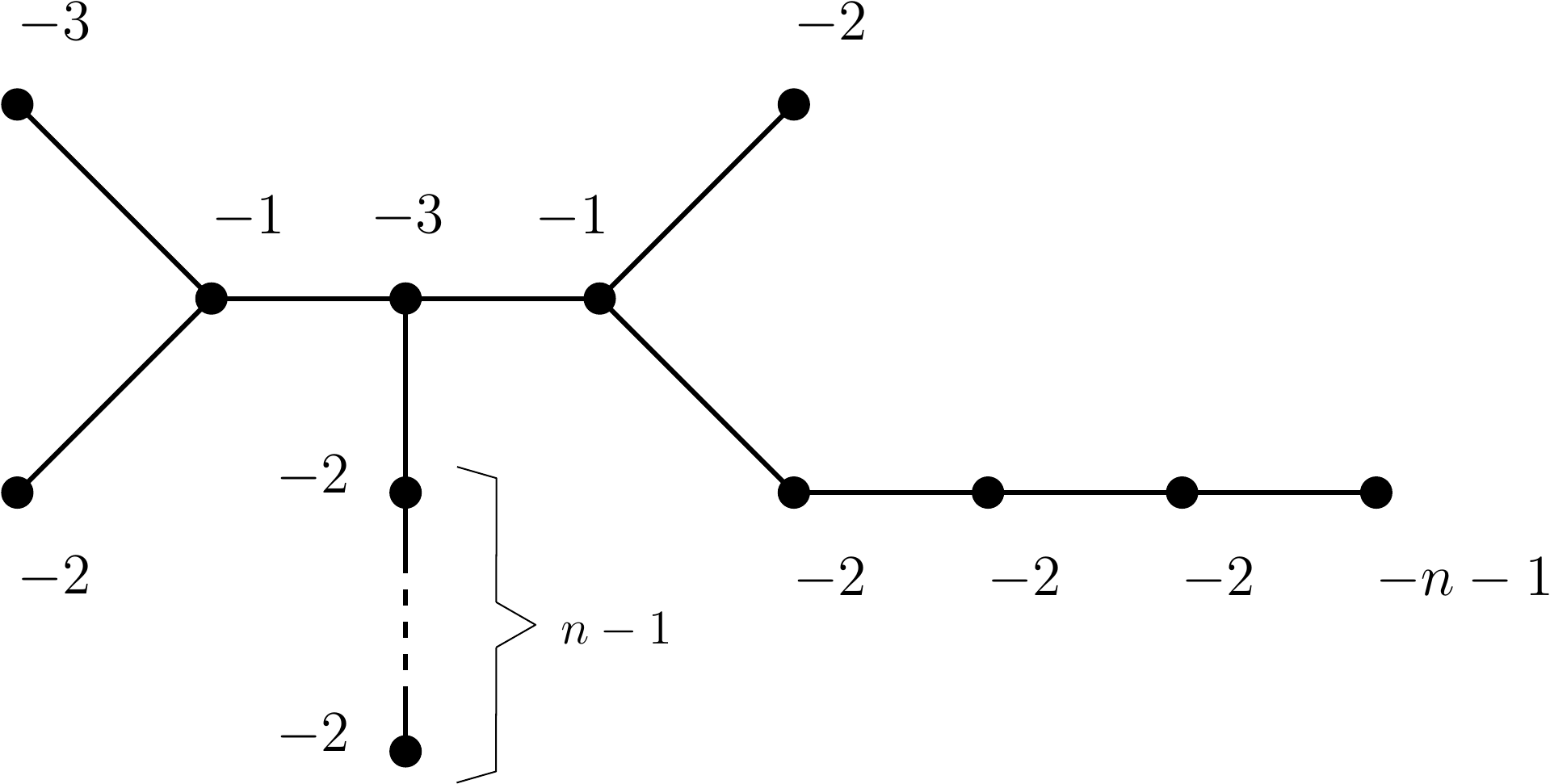}
\end{center}
\caption{Ramanujam plumbed $4$-manifold $W(n)$.}
\label{fig:plumb2}
\end{figure}

The case $n=1$ of our theorem interestingly appeared as a counter-example in the influential article of Ramanujam \cite{R71} before the Kirby calculus was introduced. This was the first example of a complex affine surface that is contractible and not analytically isomorphic to $\mathbb{C}^2$. This also provided the existence of the first exotic algebraic structure on $\mathbb{C}^3$ since $W(1) \times \mathbb{C}$ is not diffeomorphic to $\mathbb{C}^3$. For further directions about this subject, see the papers \cite{M80}, \cite{Z98} and \cite{SS05}. To the author's knowledge, this was also the first example of plumbed homology spheres bounding contractible manifolds.

We reprove most of the classical results about Seifert fibered homology spheres bounding contractible manifolds. Such spheres coincide with Brieskorn homology spheres $\Sigma(p,q,r)$ which are the links of singularities $x^p + y^q + z^r=0$. The following families were from the lists of Casson and Harer \cite{CH81}. The special cases of their list were also independently exhibited in the papers of Martin \cite{M79} and Maruyama \cite{M81}.

\begin{theorem}[Fukuhara, Maruyama, Martin, Casson-Harer]
\label{classical}
Let $p$ and $s$ be integers such that $p$ is odd. Then each Seifert fibered homology sphere $\Sigma(p,ps+1,ps +2)$ and $\Sigma(p,ps-2,ps -1)$ bounds a Mazur type contractible manifold with one $0$-handle, one $1$-handle and one $2$-handle.
\end{theorem}

The Seifert fibered homology spheres $\Sigma(2,3,13)$ and $\Sigma(2,3,25)$ were known to bound Mazur type contractible manifolds, see \cite{AK79} and \cite{F84}. Fickle proved that Stern's examples \cite{S78} bound Mazur type contractible manifolds as well. Our current perspective indicates that they also bound Po\'enaru manifolds.

\begin{theorem}
\label{classical2}
The following Seifert fibered homology spheres bound Po\'enaru type homology balls with one $0$-handle, two $1$-handles and two $2$-handles: $\Sigma(2,3,13)$, $\Sigma(2,3,25)$, $\Sigma(2,4n+1,20n+7)$, $\Sigma(3,3n+1,21n+8)$, $\Sigma(2,4n+3,20n+13)$, and $\Sigma(3,3n+2,21n+13)$ for each $n\geq 1$.
\end{theorem}

Therein, our approach simply unifies the proofs of classical and new examples of plumbed homology spheres bounding contractible manifolds and homology balls. All these examples start with the only ribbon twist knots \cite{CG86}, namely the unknot and the stevedore knot, except new manifolds presented in the paper. Their initial ribbon knot is the square knot. Note that the proofs for $\Sigma(2,3,13)$ and $\Sigma(2,3,25)$ varied from those found in \cite{S20}. Here, we cannot use the surgery description of the manifolds $\Sigma(2,3,6n+1)$ in terms of the twist knots. Using similar additional handle attachments, one can reprove that $\Sigma(2,3,7)$ and $\Sigma(2,3,19)$ bound rational homology balls, compare with \cite[Proposition 3]{AL18}. Since the proofs are addressed to the dual approach, newly constructed manifolds may or may not be diffeomorphic to classical ones. This is currently unknown to the author, but is highly unlikely to be the case for most of them.

For homology spheres, bounding homology balls is equivalent to being homology cobordant to the $3$-sphere $S^3$. Thus, they represent the trivial element in the \emph{homology cobordism group} $\Theta^3_\mathbb{Z}$. On the other hand, the algebraic structure of $\Theta^3_\mathbb{Z}$ is very complicated. The pioneering work of Dai, Hom, Stoffregen, and Truong shows that $\Theta^3_\mathbb{Z}$ admits a $\mathbb{Z}^\infty$ summand \cite{DHST18}, for more examples see \cite{KS22}. For the crucial role of $\Theta^3_\mathbb{Z}$ in topology, one may consult the references \cite{Sav02, M18, S22}.

\subsection*{Organization}
The structure of the paper is as follows. In Section~\ref{ihb}, we give the preliminaries about the essential tools for the proofs of theorems. In Section~\ref{proof}, we describe the blow down procedures for surgery diagrams of plumbed homology spheres. We first present the proof of Theorem \ref{main}. Then the proofs of Theorem \ref{maruyama} and Theorem \ref{maruyama2} are given together. Finally, we jointly prove Theorem \ref{classical} and Theorem \ref{classical2}, based on the previous work \cite{S20}. 

\subsection*{Acknowledgements}
The author would like to thank Selman Akbulut, Ronald Fintushel, Tye Lidman, Nikolai Saveliev, Jonathan Simone, Ronald Stern, and B\"ulent Tosun for valuable comments. The author is especially grateful to Marco Golla for his close interest and helpful suggestions, to his advisor \c{C}a\u gr{\i} Karakurt for his guidance and support, to Noriko Maruyama for sharing the scans of her papers, and to Bar\i \c{s} Parlatangiller for his generous help about drawings of handle diagrams. Finally, the author would like to thank the anonymous referee for their careful review. They found a gap in Lemma~\ref{integral} which is discussed in detail in Remark~\ref{rem1} and Remark~\ref{rem2}.
\section{Plumbings, Slice and Ribbon Knots, and Contractible Manifolds}
\label{ihb}

\subsection{Plumbings}
\label{plumbeds}
A \emph{plumbing graph} $G$ is a weighted tree such that each vertex $v_i$ is decorated by an integer $e_i$. A plumbing graph $G$ gives rise to a simply connected smooth $4$-manifold with boundary $X(G)=X$ which is obtained by plumbing together a collection of the $D^2$-bundles over $S^2$. Thus, the Euler number of the $D^2$-bundle corresponding to the vertex $v_i$ is given by $e_i$, see the recipe in \cite[Section 1.1.9]{Sav02}. Let $Y(G)=Y$ be the $3$-manifold which is the boundary of $X$. Then $Y$ is said to be a \emph{plumbed $3$-manifold}.

The straight lines between any two vertices resulted in the plumbing process form \emph{edges} and a linear plumbing graph having at least one edge is called a \emph{branch}. Vertices with at least three adjacent branches are called \emph{nodes}. 

Let $\vert G \vert$ denote the number of vertices of a plumbing graph $G$. The second homology group $H_2(X, \mathbb{Z})$ of the plumbed $4$-manifold $X$ is generated by vertices of $G$, and the intersection form on $H_2(X, \mathbb{Z})$ is given by the adjacency matrix 
\[ I=(a_{ij})_{i,j \in \{1,\ldots, \vert G \vert \} } \ \ \ \text{where} \ \ a_{ij} = \begin{cases} 
      e_i, & \text{if} \ v_i=v_j, \\
      1, & \text{if} \ v_i \ \text{and} \ v_j \ \text{are connected by one edge}, \\
      0, & \text{otherwise}. 
   \end{cases} \]

The plumbed $3$-manifold $Y$ is called a \emph{plumbed homology sphere} if and only if the determinant of the intersection matrix $I$ is  $\pm 1$. In this case, the corresponding plumbing graph is said to be \emph{unimodular.} When the signature of $I$ is equal to the minus of number of vertices, a plumbing graph is called \emph{negative-definite}. 

\subsection{Slice and Ribbon Knots, Contractible Manifolds and Homology Balls}

Independently in the early 1960s, Mazur \cite{M61} and Po\'enaru \cite{P60} provided the first constructions of contractible $4$-manifolds with non-trivial homology $3$-sphere boundaries (i.e. they are not diffeomorphic to the $3$-sphere $S^3$).

A \emph{Mazur manifold} is a contractible $4$-manifold built with one $0$-handle $B^4$, one $1$-handle $B^1 \times B^3$, and one $2$-handle $B^2 \times B^2$, and the algebraic intersection number of $1$- and $2$-handle is $1$ up to sign. Thus, a Mazur manifold has the simplest handlebody description for a contractible $4$-manifold other than $B^4$, i.e., a single $0$-handle. Note that the union of a $0$-handle and a $1$-handle results in $S^1 \times B^3$. We can think the latter manifold as the unknotted disk exterior of $B^4$, i.e., $B^4 \setminus \nu(D_U)$ where $U$ is the unknot in $S^3$.

We can reinterpret Mazur's original construction using ribbon knots. If a knot in $S^3$ bounds a smoothly properly embedded disk in $B^4$, then this knot and the corresponding disk are called \emph{a slice knot} and \emph{a slice disk}, respectively. Using Morse theory, we can isotope any slice disk such that their level sets with respect to the radial function on $B^4$ are links, except in the case of singularities. Moreover, we have three types of singularities, namely minima, saddles, and maxima, corresponding to $0$-, $1$-, and $2$-handles of $2$-dimensional slice disks. 

If a slice disk does not contain $2$-handles, then we call it a \emph{ribbon disk}, and the corresponding knot is said to be a \emph{ribbon knot}. We call the minimum number of $1$-handles for a ribbon disk the \emph{fusion number}. Since the Euler characteristic of a disk is $1$, the fusion number completely determines the number of $0$-handles for ribbon disks, which is equal to the fusion number plus $1$. The basic examples of ribbon knots are the stevedore knot and the square knot (i.e. connected sum of the left- and right-handed trefoils).

Now, we can change the role of the unknot in Mazur's construction with any ribbon knot to construct new manifolds. We attach a single $2$-handle to a ribbon disk exterior to obtain a \emph{Po\'enaru type homology ball}. Therefore, a Po\'enaru manifold is a natural generalization of a Mazur manifold.

\begin{lemma}
\label{integral}
Let $Y$ be the $3$-manifold obtained by $0$-surgery on a ribbon knot in $S^3$ of fusion number $n \geq 1$. Then any homology sphere obtained by an integral surgery on a knot in $Y$ bounds a Po\'enaru type homology ball with one $0$-handle, $n+1$ $1$-handles, and $n+1$ $2$-handles. Moreover, if the initial ribbon knot is the unknot, then the resulting homology sphere bounds a Mazur type contractible manifold with one $0$-handle, one $1$-handle, and one $2$-handle.
\end{lemma}

\begin{proof}

Let $K$ be a ribbon knot in $S^3$ of fusion number $n \geq 0$ and let $D$ be a ribbon disk in $B^4$ with $\partial D = K$ . Consider the ribbon disk exterior $X = B^4 \setminus \nu(D)$. Then we have $\partial X = Y$ since $\nu(D)$ is diffeomorphic to $D \times B^2$. In addition, the ribbon disk exterior $X$ has a handle decomposition with one $0$-handle, $n+1$ $1$-handles, and $n$ $2$-handles, see \cite[Section~6.2]{GS99}.

An integral surgery on any $J \subset Y$ corresponds to attaching a $k$-framed $2$-handle $B^2 \times B^2$ to $X$ along $J$, and produces a new $4$-manifold $W$. Here, we can place $J$ in $S^1 \times D^2 \subset Y$, so the framing of the additional $2$-handle is fixed due to Akbulut's carving notion, see \cite[Section~1.4]{A16}. 

Let $Y'$ be the homology sphere obtained by an integral surgery on $J \subset Y$. Assume that the additional $2$-handle wraps $m$ times around the ribbon disk exterior $X$. Since $Y'$ is a homology sphere, the determinant of the linking matrix is $0 \cdot k - m^2 = \pm 1$, so $m$ must be $\pm 1$. Thus, $W$ is always a homology ball since $H_1 (W) = 0$. 

If the fusion number $n$ is zero, then $K$ is the unknot, and in this case $J$ is freely homotopic to the meridian of $K$. Therefore, by van Kampen's theorem, $W$ is a simply-connected homology ball, hence a Mazur type contractible manifold by the classical theorems of Hurewicz and Whitehead.
\end{proof}

\begin{remark}
\label{rem1}
The resulting Po\'enaru type homology ball \emph{might not be} contractible: Consider the ribbon disk exterior for the square knot. Then it has the fundamental group given by the fundamental group of the trefoil knot exterior. Now, one attaches the additional $2$-handle along a knot $J \subset Y$ which represents meridian times longitude in $\pi_1 (X)$. Then $W$ is a homology ball with the fundamental group equal to that of $(+1)$-surgery on the trefoil. Therefore, $\pi_1 (W)$ is either the binary icosahedral group or a central $\mathbb{Z}$-extension of the $(2,3,7)$ triangle group, depending on the choice of the meridian-longitude pair. These groups are finite and infinite respectively, see for instance Milnor's article \cite{Mil75}. 
\end{remark}

\begin{remark}
\label{rem2}
The resulting Po\'enaru type homology ball \emph{might be} contractible: if $J$ is the meridian of the knot $K$, then they together form $(1/k)$-surgery along $K$ in $S^3$, then $W$ with $\partial W = S^3_{1/k} (K)$ is always a contractible manifold by the classical result of Gordon \cite{Gor75}. However, an integral surgery on a random knot $J$ in $S_0^3 (K)$ may still produce $S^3_{1/k} (K)$, and in this case $J$ may not be homotopic to the meridian of $K$ or may not normally generate $\pi_1 (X)$.
\end{remark}

\begin{remark}
Presumably, there would be handle cancellations in the handle decomposition of a ribbon disk exterior, or a Po\'enaru manifold. Therefore, we may address contractible manifolds having fewer numbers of $1$- and $2$-handles. 
\end{remark}

Our next ingredient is the following trick of Akbulut and Larson coming from the proof of their main theorem \cite[Theorem~1]{AL18}. 

\begin{definition}
\normalfont The \emph{Akbulut-Larson trick} is an observation about describing the iterative procedure for passing from the surgery diagram of a homology sphere to a consecutive one by using a single blow up with an isotopy, see Figure~\ref{fig:trick}.

\begin{figure}[ht]
\centering
\includegraphics[width=0.7\columnwidth]{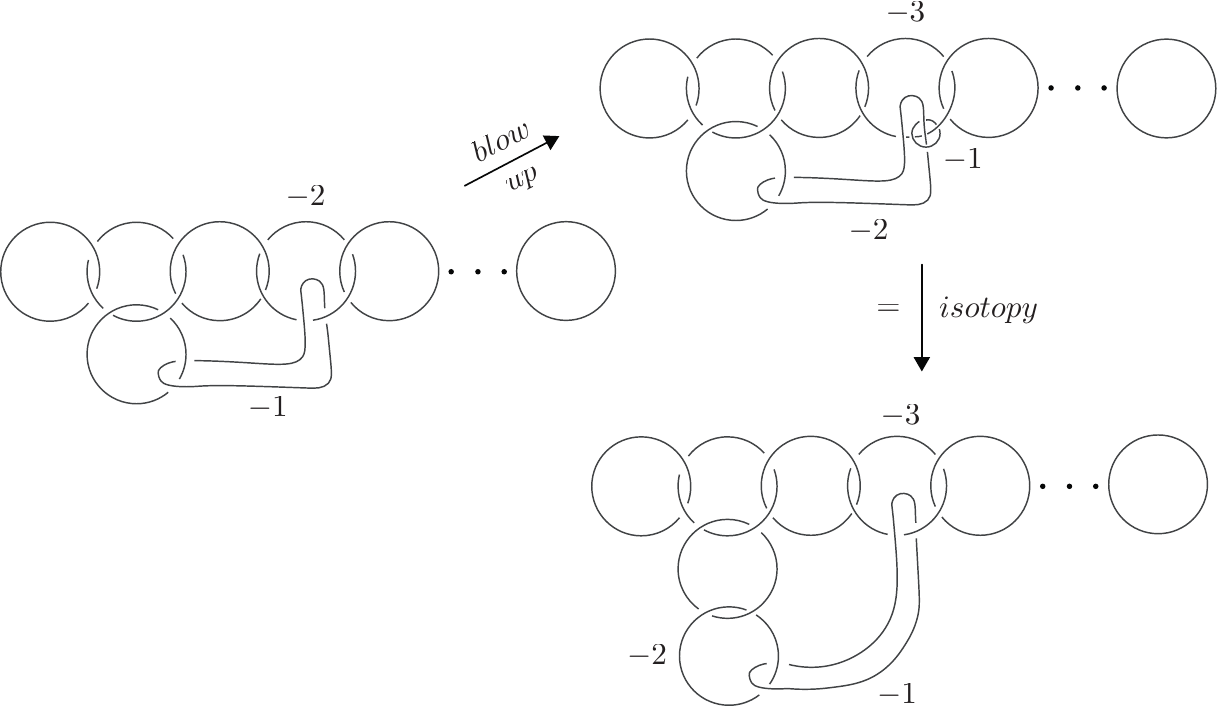}
\caption{The Akbulut-Larson trick.}
\label{fig:trick} 
\end{figure}

\end{definition}

\section{Proofs of Theorems}
\label{proof}
In this section, we prove all our theorems using the background in Section~\ref{ihb}. We shall start with the proof of Theorem~\ref{main}.

\begin{proof}[Proof of Theorem~\ref{main}]
Due to the combinatorial approach in \cite[Chapter 5.21]{EN85}, it can be easily checked that the determinant of intersection matrix associated to the plumbing graph for $W(n)$ is $\pm 1$ according to the parity of $n$, so $\partial W(n)$ is a homology sphere for each $n \geq 1$. Also since the plumbing graph for $W(n)$ has three nodes and seven branches, it is not a Seifert fibered homology sphere, see \cite[Chapter 2.7]{EN85}.

Here, we use the dual approach by giving integral surgeries from its plumbing graph displayed in Figure~\ref{fig:plumb2}. We actually show that $\partial W(n)$ is obtained by $(-1)$-surgery on a knot in $Y$ where $Y$ is $0$-surgery on the square knot. Its surgery diagram corresponding to the case $n=1$ appears in Figure~\ref{fig:w}. The dark black $(-1)$-framed component gives the necessary integral surgery to $Y$ after applying blow downs several times. Then the general family is obtained by applying the Akbulut-Larson trick successively. Therefore, we finish the proof by using Lemma~\ref{integral} since the blow down operation does not change the boundary $3$-manifold.

\begin{figure}[ht]
\vskip\baselineskip
\centering
\includegraphics[width=0.8\columnwidth]{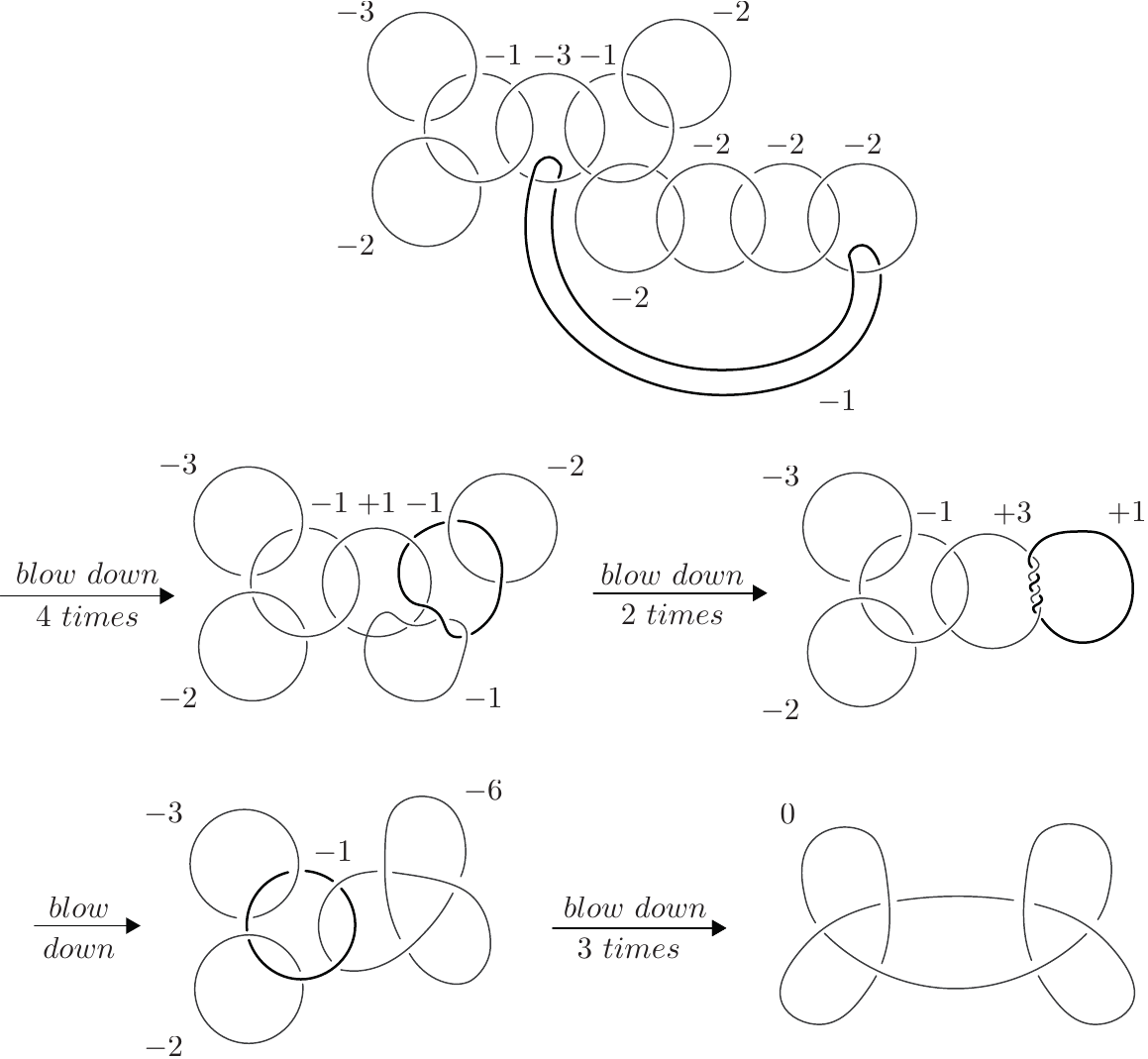}
\caption{$(-1)$-surgery from $\partial W(1)$ to $Y$.}
\label{fig:w}
\end{figure}

\end{proof}

Our next proof is about the boundary of the Maruyama plumbed $4$-manifold as well as its companion introduced in the paper.

\begin{proof}[Proof of Theorem~\ref{maruyama} and Theorem~\ref{maruyama2}]

Following the recipe in \cite[Chapter 5.21]{EN85} again, it is straightforward to check that the determinant of the intersection matrix associated to the plumbing graph for $X'(n)$ is $- 1$ for each $n \geq 1$, so $\partial X'(n)$ is a homology sphere.

Again we use the dual approach for describing the additional integral surgeries from their plumbing graphs shown in Figure~\ref{fig:plumb}. For the case $n=1$, see the analogous proof of \cite[Theorem 1.2]{S20}. Assume that $n \geq 2$. In particular, we prove that there is a $(-1)$-surgery on a knot in $\partial X'(n)$ to $Y$ where $Y$ is $0$-surgery on the unknot. In Figure~\ref{fig:x1}, we draw the blow down sequences explicitly and we eventually reach the $3$-manifold $Y$. Hence, the rest of the proof is using Lemma~\ref{integral} and the fact that the blow down operation keeps the homeomorphism type of the boundary $3$-manifold the same. The proof for $\partial X(n)$ is quite similar to the previous one and it is clearly displayed in Figure~\ref{fig:x1}. For $n=1$, one can again consult the proof of \cite[Theorem~1.2]{S20}.

\begin{figure}[ht]
\centering
\includegraphics[width=0.95\columnwidth]{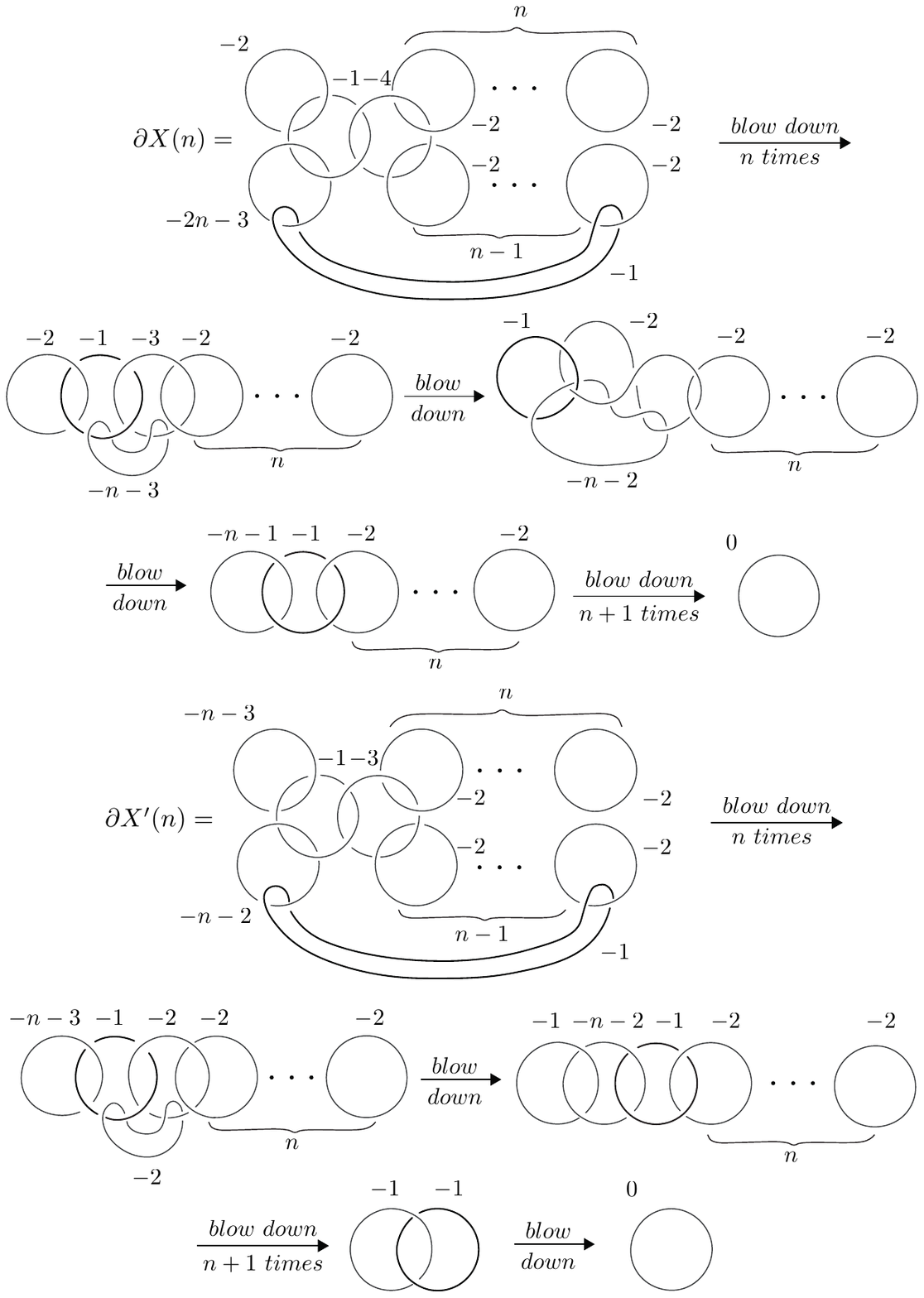}
\caption{$(-1)$-surgeries respectively from $\partial X(n)$ and $\partial X'(n)$ to $Y$.}
\label{fig:x1}
\end{figure}

\end{proof}

The discussion in the following remark is suggested by Marco Golla.

\begin{figure}[ht]
\centering
\includegraphics[width=0.32\columnwidth]{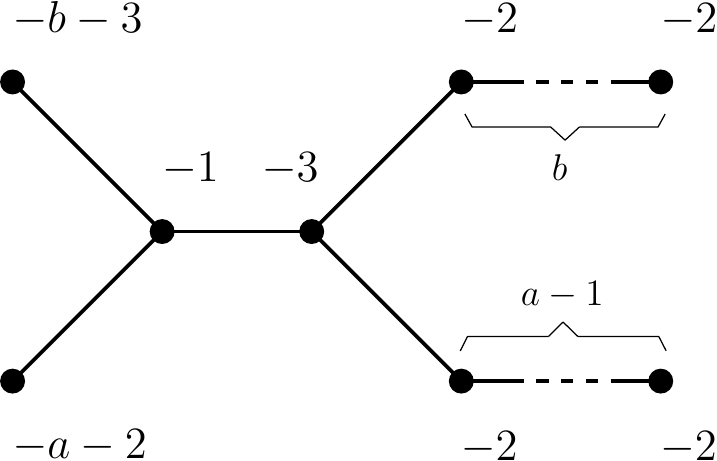}
\caption{The two-parameter plumbed manifold $X'(a,b)$.}
\label{fig:plumb5}
\vskip\baselineskip
\end{figure}

\begin{remark}
One may recognize that the upper part of the plumbing graph of $X'(n)$ behaves freely in the proof, so one might think that we obtain a two-parameter family of homology spheres, say $X'(a,b)$, displayed in Figure~\ref{fig:plumb5}. Then the determinant of its intersection matrix is $(-1)^{a-b-1}(a-b-1)^2$. Thus, its boundary $\partial X'(a,b)$ is a homology sphere if and only if $a=b$ or $a=b+2$. However, in this case, $X'(b,b)$ and $X'(b+2,b)$ are essentially the same as our initial $4$-manifold $X'(n)$.
\end{remark}

We finally reprove the classical results about Seifert fibered homology spheres.

\begin{proof}[Proof of Theorem~\ref{classical} and Theorem~\ref{classical2}]
Using the procedure in \cite[Example 1.17]{Sav02}, it can be simply shown that $\Sigma(2,3,13)$ and $\Sigma(2,3,25)$ are the boundary of the negative-definite unimodular plumbing graphs shown in Figure~\ref{fig:pl2}, respectively.

\begin{figure}[ht]
\centering
\includegraphics[width=0.7\columnwidth]{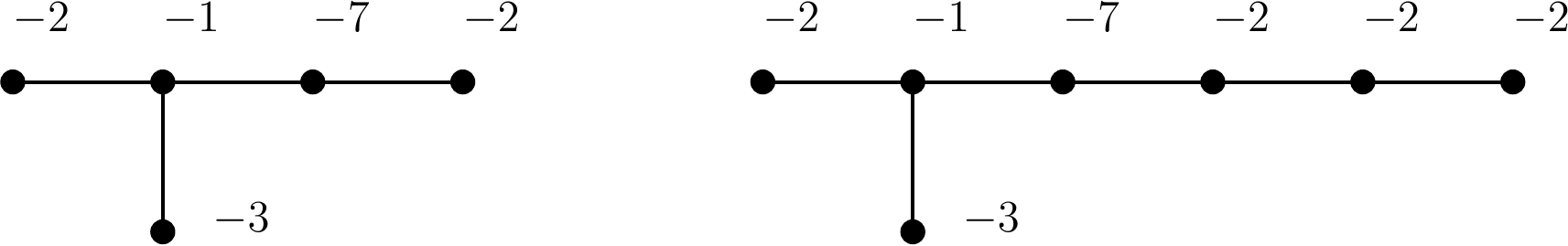}
\caption{The plumbing graphs of $\Sigma(2,3,13)$ and $\Sigma(2,3,25)$}
\label{fig:pl2}
\end{figure}

In Figure~\ref{fig:two}, we draw the additional $(-2)$-framed components in their surgery diagrams. Blowing down the $(-1)$-framed dark black components $5$- and $7$-times respectively, we ultimately reach the $3$-manifold $Y$ where $Y$ is $0$-surgery on the stevedore knot. Applying Lemma~\ref{integral}, we finish the proofs of these cases. The explicit procedures for blow down sequences are left to readers as exercises. The proofs for Stern's families listed in Theorem~\ref{classical2} are identical within \cite{S20}. One can consult the handle diagrams appearing in \cite[Theorem~1.2]{S20}. 

\begin{figure}[ht]
\vskip\baselineskip
\centering
\includegraphics[width=0.84\columnwidth]{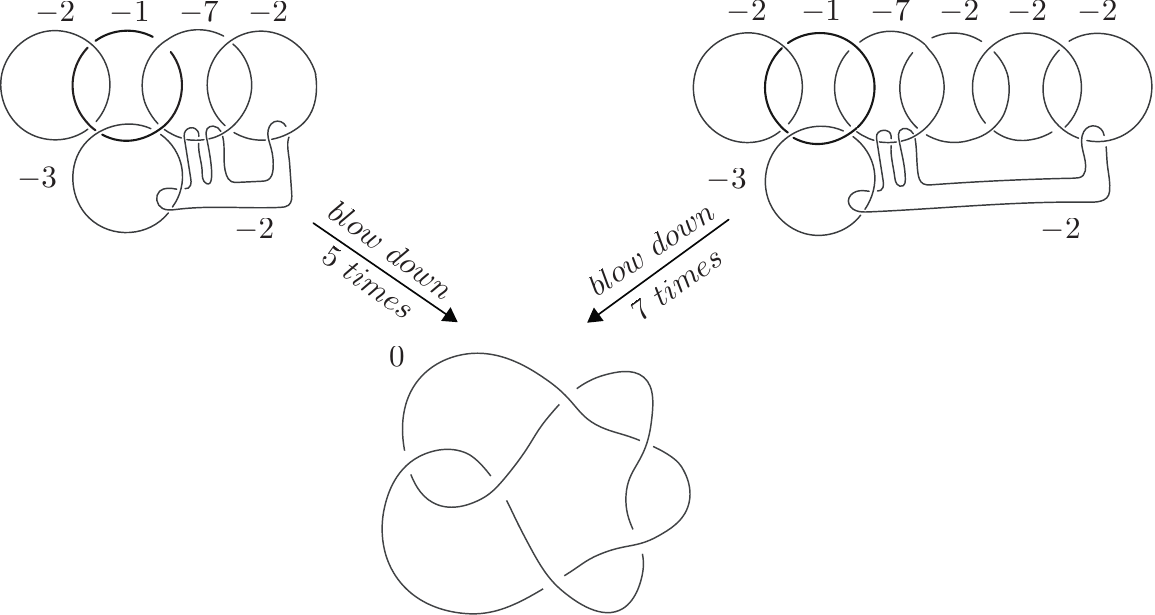}
\caption{$(-2)$-surgeries from $\Sigma(2,3,13)$ and $\Sigma(2,3,25)$ to $Y$.}
\label{fig:two}
\vskip\baselineskip
\end{figure}

Next, we consider families of Casson and Harer. For odd values of the integer $p$, recall that $\Sigma(p,ps+1,ps+2)$ and $\Sigma(p,ps-2,ps-1)$ are the boundaries of the negative-definite unimodular plumbing graphs displayed in Figure~\ref{fig:pl} from the left to the right \cite{CH81}. To complete the proof by applying Lemma~\ref{integral}, we similarly address the dual approach by giving integral surgeries from their plumbing graphs. In fact, we show that $\Sigma(p,ps+1,ps+2)$ and $\Sigma(p,ps-2,ps-1)$ are both obtained by $(-1)$-surgery on a knot in $Y$ where $Y$ is $0$-surgery on the unknot.

\begin{figure}[ht]
\centering
\includegraphics[width=0.75\columnwidth]{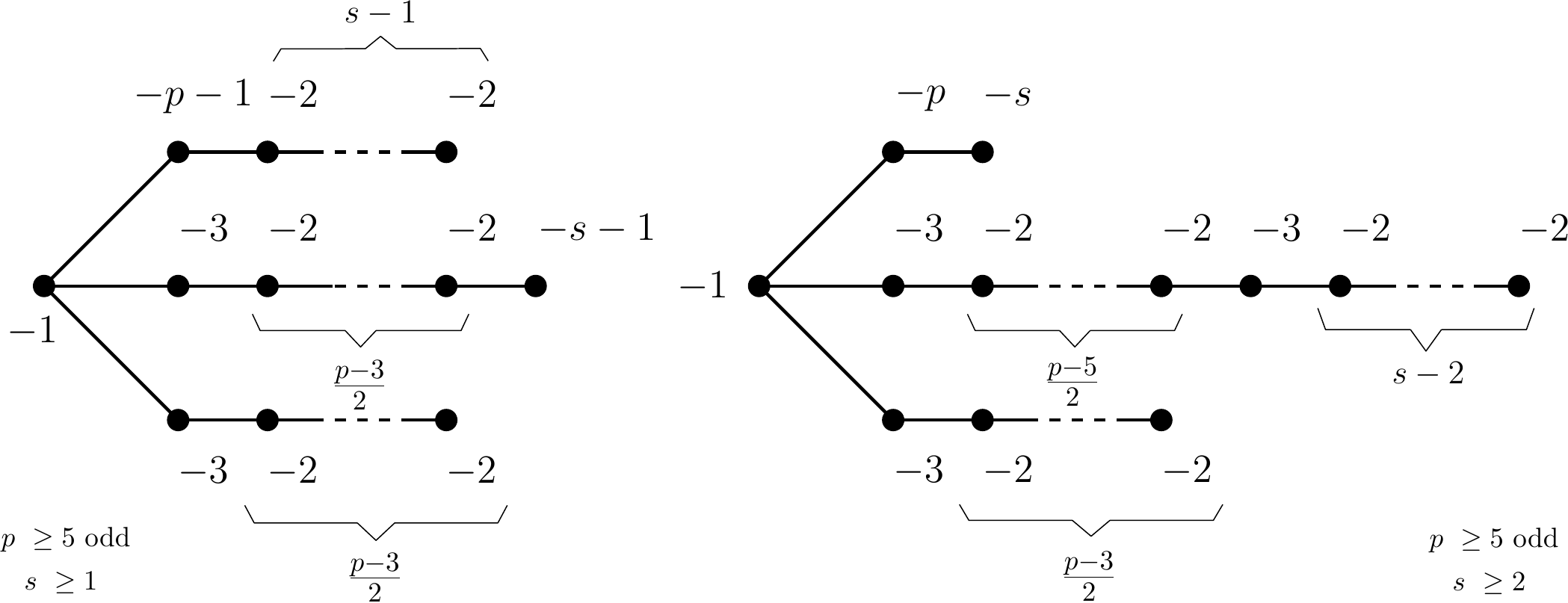}
\caption{The plumbing graphs of families of Casson and Harer.}
\label{fig:pl}
\vskip\baselineskip
\end{figure}

Assume $p=3$ for both families, so that we consider $\Sigma(3,3s+1,3s+2)$ and $\Sigma(3,3s-2,3s-1)$. The case $s=1$ for the second one results in $\Sigma(3,1,2) = S^3$. The remaining elements are essentially the same as the former family, and the proof of $\Sigma(3,3s+1,3s+2)$ is already in \cite{S20} with the same technique. Henceforth, we suppose that $p \geq 5$ odd and $s \geq 1$ for the former family, and $p \geq 5$ odd and $s \geq 2$ for the latter one. The surgery diagrams corresponding to the first elements of the plumbing graphs show in Figure~\ref{fig:f1} and Figure~\ref{fig:f2} respectively. Again the extra dark black $(-1)$-framed components indicate the required surgeries to $Y$ by following the blow down procedures. Then the whole families are obtained by applying the Akbulut-Larson trick successively.

\begin{figure}[ht]
\centering
\includegraphics[width=0.95\columnwidth]{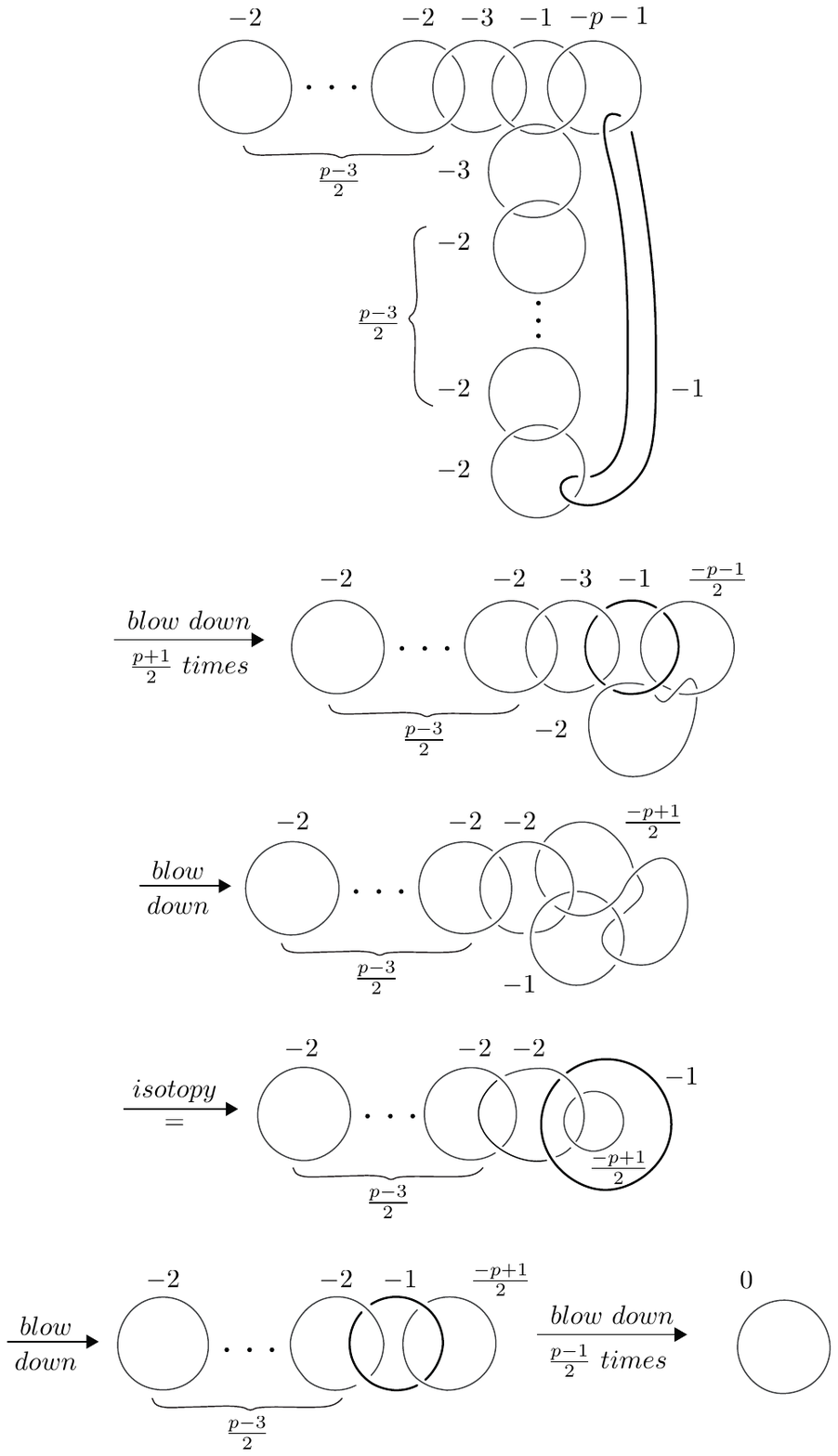}
\caption{$(-1)$-surgery from $\Sigma(p,p+1,p+2)$ to $Y$.}
\label{fig:f1}
\end{figure}

\begin{figure}[ht]
\centering
\includegraphics[width=0.95\columnwidth]{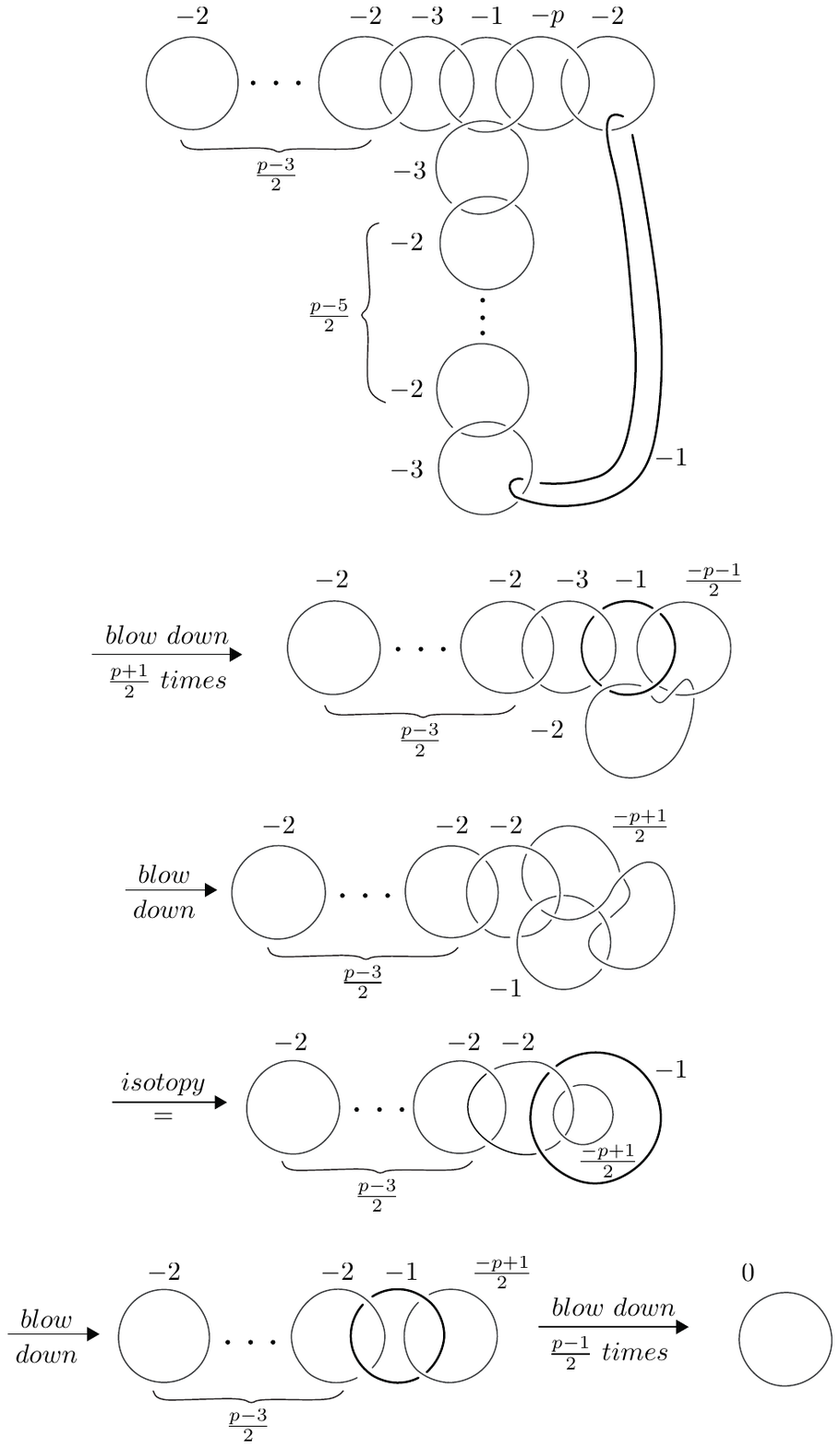}
\caption{$(-1)$-surgery from $\Sigma(p,2p-2,2p-1)$ to $Y$.}
\label{fig:f2}
\end{figure}

\end{proof}

\clearpage
\bibliography{plumbedhomspheres}
\bibliographystyle{amsalpha}

\end{document}